\RequirePackage{fix-cm}
\documentclass[envcountsect]{svjour3}                     
\smartqed  
\usepackage{graphicx}
\usepackage{amsmath}%
\usepackage{amstext}%
\usepackage{amssymb}%
\usepackage{amsxtra}
\usepackage{epsfig}%
\usepackage{cite}
\usepackage{tikz}
\usepackage{caption}
\usepackage{rotating,booktabs}
\usepackage{multirow}
\usepackage{booktabs}
\usepackage{floatrow}
\usepackage{listings}
\spnewtheorem*{condition}{Condition}{\bf}{\it}
\spnewtheorem{assumption}{Assumption}[section]{\bf}{\it}
\spnewtheorem{pb}{Problem}[section]{\bf}{\it}
\spnewtheorem{algorithm}{Algorithm}[section]{\bf}{\rm}
\newcommand{\Hi}{\mathcal{H}}
\newcommand{\R}{\mathbb{R}}%
\newcommand{\ep}{\varepsilon}%

\renewcommand{\>}{\right\rangle}
\newcommand{\<}{\left\langle}
\newcommand{\zer}{\mathrm{zer}}
\newcommand{\Gr}{\mathrm{Gr}}
\DeclareMathOperator*\dist{dist}%
\DeclareMathOperator*\inte{int}%
\DeclareMathOperator*\dom{dom}%
\DeclareMathOperator*\ran{ran}%
\DeclareMathOperator*\prox{prox}%
\DeclareMathOperator*\proj{proj}%
\DeclareMathOperator*\argmin{argmin}

\begin{document}

\title{Generalized Forward-Backward Splitting with Penalization for Monotone Inclusion Problems\thanks{This work  is partially supported
		by the Thailand Research Fund under the project RSA5880028.}
	}

\titlerunning{Generalized Forward-Backward Splitting with Penalization}        

\author{Nimit Nimana   \and
	Narin Petrot
}


\institute{     Nimit Nimana \at
	Department of Mathematics, Faculty of Science, Khon Kaen University, Khon Kaen, 40002 Thailand \\
	\email{nimitni@kku.ac.th} 
	\and 
	Narin Petrot \at
	Department of Mathematics, Faculty of Science, Naresuan University, Phitsanulok,  65000, Thailand, Center of Excellence in Nonlinear Analysis and Optimization, Faculty of Science, Naresuan University,
	Phitsanulok, 65000, Thailand\\
	\email{narinp@nu.ac.th}           
}

\date{Received: date / Accepted: date}

\maketitle

\begin{abstract}
We introduce a generalized forward-backward splitting method with penalty term for solving monotone inclusion problems involving the sum of a finite number of maximally monotone operators and the normal cone to the nonempty set of zeros of another maximally monotone operator. We show weak ergodic convergence of the generated sequence of iterates to a solution of the considered monotone inclusion problem, provided that the condition corresponding to the Fitzpatrick function of the operator describing the set of the normal cone is fulfilled. Under strong monotonicity of an operator, we show strong convergence of the iterates. Furthermore, we utilize the proposed method for minimizing a large-scale hierarchical minimization problem concerning the sum of differentiable and nondifferentiable convex functions subject to the set of minima of another differentiable convex function. We illustrate the functionality of the method through numerical experiments addressing constrained elastic net and generalized Heron location problems.
\keywords{Monotone operator \and monotone inclusion \and hierarchical optimization \and forward-backward algorithm}

\subclass{47H05 \and 65K05 \and 65K10 \and 90C25}
\end{abstract}

\section{Introduction}
\label{intro}
Let $\Hi$ be a real Hilbert space with the norm and inner product given by $\|\cdot\|$ and $\langle\cdot,\cdot\rangle$, and $m$ be a positive integer.  We firstly recall the problem of minimizing an additive component function:
\begin{eqnarray}\label{opt-old}%
\begin{array}{ll}
\textrm{minimize}\indent \sum_{i=1}^mf_i(x)\\
\textrm{subject to}\indent x\in X,
\end{array}%
\end{eqnarray}
where, for all $i=1,\ldots,m$, $f_i:\Hi\to (-\infty,+\infty]:=\R\cup\{+\infty\}$ is a proper convex lower semicontinuous objective function and $X$ is a nonempty closed  convex subset of $\Hi$. Due to a practical point of view, many real-world problems can be formulated in the form of the problem (\ref{opt-old}). For instance, in large-scale machine learning (where $m$ is very large), each component function $f_i$ will measure the sufficiency of the model's output corresponding to a parameter $x\in\Hi$ and an observed data indexed by $i$. In this situation, minimizing the sum of these component functions will give a parameter $x\in\Hi$ in which it 
fits the observed data.  A classical example is the least squares regression
\begin{eqnarray*}\label{opt-lqr}%
	\begin{array}{ll}
		\textrm{minimize}\indent \sum_{i=1}^m\left(\langle a_i,x\rangle-b_i\right)^2\\
		\textrm{subject to}\indent x\in \Hi,
	\end{array}%
\end{eqnarray*} 
where $a_i\in \Hi$ and $b_i\in\R$ are given data, for all $i=1,\ldots,m$. We refer to Bottou, Curtis and Nocedal \cite{BCN16} for an attractive review and \cite{J15,SLB17,SZ13} for more contributions on large-scale machine learning.

It is well known that algorithms for approximating a solution of (\ref{opt-old}) may involve the metric projection onto the feasible set $X$. However, in some situations such set $X$ is not simple enough so that  the projection cannot be easily implemented. In order to deal with this situation, Attouch et al. \cite{ACP11} proposed an  exterior penalization scheme involving the gradient of a corresponding constrained function instead of computing the metric projection onto constrained sets directly.  
Although, sometimes  the set $X$ is simple enough so that the metric projection has a closed-form expression,  it may happen that the computation is very costly and time-consuming. For instance, when $\Hi=\R^n$, it is well known that the set $X:=\{x\in\R^n:Ax=b\}$, where $A$ is an $r\times n$ matrix  (with $r<n$) with linearly independent rows and $b\in\R^n$, has infinitely many elements and the metric projection onto this set is $\proj_X(x):=x-A^\top(AA^\top)^{-1}(Ax-b)$, which is not easy to invert. Of course, one can take the function $g(x)=\frac{1}{2}\|Ax-b\|^2$ so that $X=\argmin g$ and the gradient is $\nabla g(x)= A^\top(Ax-b)$, which is simpler to compute than the metric projection.  It is worth noting that the consideration of particular structure by using the constrained function has been applied in several situations such as optimal control, partial differential equation, signal and image processing, see \cite{ACP11,ACP11-2,P12,NP13} for more  insight details. These advantages naturally motivate us to consider the particular structure of the constrained set  $X=\argmin g$,  which leads us to consider the following hierarchical minimization problem:

\begin{eqnarray}\label{opt}%
\begin{array}{ll}
\textrm{minimize}\indent \sum_{i=1}^mf_i(x)\\
\textrm{subject to}\indent x\in\argmin g,
\end{array}%
\end{eqnarray} 
where, for all $i=1,\ldots,m$, $f_i:\Hi\to (-\infty,+\infty]$ is a proper convex lower semicontinuous objective function, and $\argmin g$ is the set of minima of a  convex (Fr\'echet) differentiable  function $g:\Hi\to \R$, which we will assume that it is nonempty and $\min g=0$. Another approach to motivate  problem (\ref{opt}) in the context of nonautonomous multiscaled differential inclusion is due to \cite{AC10}. We refer the reader to \cite{ACP11,ACP11-2,CNP16,NP13,P12, BCN17,BCN17-2} for a rich literature devoted to problem  (\ref{opt}).

Assume that the solution set of the problem (\ref{opt}) is nonempty and  some qualification conditions hold, for instance,
\begin{eqnarray}\label{QC-opt}%
\argmin g\cap\bigcap_{i=1}^m\inte\dom(f_i)\neq\emptyset
\end{eqnarray}
(for more qualification conditions, see \cite[Proposition 27.8]{BC11}). Then, problem (\ref{opt}) is equivalent to the following problem:
find $x\in \Hi$ such that
\begin{eqnarray}\label{MIP-opt}%
0\in \sum_{i=1}^m\partial f_i(x)+N_{\argmin g}(x).
\end{eqnarray}
Note that the nonempty subdifferential of a proper convex lower semicontinuous function is a maximally monotone operator and the set of minima of such function is the set of zeros of the subdifferential. Due to the equivalence of problems (\ref{opt}) and (\ref{MIP-opt}), in this paper, we shall deal with the following monotone inclusion problem.
\begin{pb}[Monotone Inclusion Problem (MIP)] \label{mip}Find $x\in \Hi$ such that
	\begin{eqnarray}\label{MIP}%
	0\in \sum_{i=1}^mA_i(x)+B(x)+N_{\zer(C)}(x),
	\end{eqnarray}
	where, for all $i=1,\ldots,m$, $A_i:\Hi\rightrightarrows \Hi$ is a maximally monotone operator and $B,C: \Hi\to \Hi$ are cocoercive operators. 
\end{pb}

Note that if $A_i=\partial f_i$, for all $i=1,\ldots,m$, $B(x)=0$ for all $x\in\Hi$ and $C=\nabla g$, then the monotone inclusion problem (\ref{MIP-opt}) is a special case of MIP. Moreover, by using some suitable qualification condition (e.g., (\ref{QC-opt})),  the hierarchical minimization problem (\ref{opt}) is also a particular case of MIP; see Section 4 for further details.

We let the set of all zeros of the operator $C$ be denoted by $\zer(C):=\{z\in \Hi:0= C(z)\}$,  and let the normal cone to the set $\zer(C)$ be denoted by $N_{\zer(C)}$. We shall assume from now on that $$\zer(\sum_{i=1}^mA_i+B+N_{\zer(C)})\neq\emptyset.$$

In order to approximate a solution of MIP, we remark that Bo\c{t} and Csetnek \cite{BC14-penalty} investigated a particular situation of MIP of the form
	\begin{eqnarray}\label{MIPm=1}%
	0\in A(x)+B(x)+N_{\zer(C)}(x),
	\end{eqnarray}
(say, $A:=A_1$). They proposed a forward-backward algorithm of penalty type for solving  problem (\ref{MIPm=1}). In each iteration, the algorithm performs a forward step with the operator $B$ together with the penalization term with respect to $C$ and a backward step by the resolvent of $A$, that is,
$$x_{k+1}:=J_{\alpha_kA}\left(x_{k}-\alpha_kB(x_k)-\alpha_k\beta_kC(x_k)\right) \indent\forall k\geq1,$$ 
where $x_1\in \Hi$ is arbitrarily chosen. To guarantee the convergence of the sequence generated by their proposed algorithm, they introduced a condition formulated by using the Fitzpatrick function associated with the operator $C$  (see Assumption \ref{assumption-IFB-P1} (H2)).  In their paper, they also presented the forward-backward-forward algorithm, which is known as Tseng's type algorithm with penalty term, for solving  (\ref{MIPm=1}) in the case where $B$ and $C$ are monotone and Lipschitz continuous.  As a continuation of these developments, in \cite{BC14-viet}, the same authors focused on solving a monotone inclusion problem involving linearly composed and parallel-sum monotone operators and the normal cone to the set of zeros of another monotone and Lipschitz continuous operator.  An application on image inpainting has been also presented.  In \cite{BC16}, the same authors considered the penalty type algorithm with inertial effect of the one proposed in  \cite{BC14-penalty} for solving  (\ref{MIPm=1}). In a similar fashion, Banert and Bo\c{t}  \cite{BB15} proposed the backward penalty type algorithm for solving  (\ref{MIPm=1}) in the case where $C$ is a maximally monotone operator.

On the other hand, Passty \cite{P79} considered a backward splitting method for solving MIP in the case where  $B=C=0$. The proposed algorithm can be read as follows:
let $x_1\in \Hi$ be chosen arbitrarily. For each $k\geq1$, one sets $\psi_{0,k}:=x_k$. One computes
\begin{eqnarray*}\psi_{i,k}:=J_{\alpha_kA_i}(\psi_{i-1,k}) \indent \textrm{ for all } i=1,\ldots,m,
\end{eqnarray*}
and put $$x_{k+1}:=\psi_{m,k}.$$
The advantage of this method is that one can evaluate the resolvents $J_{A_i}$, $i=1,\ldots,m$ individually  rather than computing the resolvent of the sum $\sum_{i=1}^mA_i$, which is quite difficult to invert in general.

Motivated by these methods, we propose a generalized forward-backward scheme with penalization term for solving MIP. By means of the condition involving the Fitzpatrick function, we prove weak ergodic convergence of the generated sequence to a solution of  MIP. Furthermore, assuming that one of the operators $A_i$ is strongly monotone, we prove strong convergence result for the generated sequence to the unique solution of  MIP. As our primary convince, convergence results for hierarchical large-scale minimization problems will be also discussed.

This paper is organized as follows: Section 2 contains some notations and useful  tools. The main algorithm and its convergence results are presented in Section 3. In Section 4, we propose a direction to convergence for hierarchical large-scale minimization problems. Finally, numerical examples are illustrated in Section 5.

\section{Preliminaries}

For convenience we present here some notations which are used throughout this paper, the reader may consult \cite{BC11,Z02} for further details.  The strong convergence and weak convergence of a sequence
$\{x_k\}_{k=1}^\infty$ to $x\in \Hi$ are denoted by $x_k\to x$ and $x_k\rightharpoonup x$, respectively.

Let $A:\Hi\rightrightarrows \Hi$ be a set-valued operator. We denote by $\Gr(A) :=\{(x,u)\in \Hi\times \Hi:u\in Ax\}$ its graph, by  $\mathrm{dom}(A) :=\{x\in \Hi:Ax\neq\emptyset\}$ its domain, and by $\ran(A) :=\{u\in \Hi:\exists x \in \Hi, u\in Ax\}$ its range. The set-valued operator $A$ is said to be \textit{monotone} if 
$\langle x-y,u-v\rangle\geq0,$ for all $(x,u), (y,v)\in \Gr(A)$, and it is called \textit{maximally monotone} if its graph is not properly contained in the graph of any other monotone operators. Moreover, $A$ is said to be {\it strongly monotone} with modulus $\alpha>0$ if $A-\alpha I$ is monotone, i.e., $\langle x-y,u-v\rangle\geq\alpha\|x-y\|^2$ for all $(x,u),(y,v)\in\Gr(A)$. Note that if $A$ is maximally monotone, then $\zer(A)$ is a convex closed set. Notice that if $A$ is maximally monotone and strongly monotone, then $\zer(A)$ is a nonempty set and is a singleton. Furthermore, for a maximally monotone operator $A$, we have
\begin{eqnarray}\label{zeromaximal}z \in \zer(A) \textrm{ if and only if } \langle u-z, w\rangle \geq 0 \textrm{ for all } (u, w) \in \Gr(A).
\end{eqnarray}

We denote by $I:\Hi\to \Hi$  the identity operator on $\Hi$.  For a set-valued operator $A:\Hi \rightrightarrows \Hi$, we define the {\it resolvent} of $A$, $J_{A}:\Hi\rightrightarrows \Hi$, by
$J_{A}:=(I+A)^{-1}$. It is well known that if $A$ is maximally monotone, then the resolvent of $A$ is single-valued.

Given $\mu>0$, the operator $A:\Hi\to \Hi$ is said to be {\it cocoercive} (or {\it inverse strongly monotone}) with parameter $\mu$ if $\langle x-y,Ax-Ay\rangle \geq\mu\|Ax-Ay\|^2$ for all $x,y\in \Hi$. 

For a monotone operator $A:\Hi\rightrightarrows \Hi$, the {\it Fitzpatrick function} \cite{F88} associated to the monotone operator $A$, $\varphi_A :\Hi\times \Hi\to(-\infty,+\infty]$, is defined by
$$\varphi_A(x,u):= \sup_{(y,v)\in\Gr(A)}\{\langle x,v\rangle+\langle y,u\rangle-\langle y,v\rangle\}.$$
Note that $\varphi_A$ is a convex  lower semicontinuous function. In addition, if $A$ is maximally monotone, then
$$\varphi_A(x,u)\geq \langle x,u \rangle\indent \textrm{ for all } (x,u)\in \Hi\times \Hi$$
and the equality holds provided that $(x,u)\in\Gr(A)$.

For a function $f:\Hi\to (-\infty,+\infty]$ we denote by $\mathrm{dom}(f):=\{x\in \Hi:f(x)<+\infty\}$ its {\it effective domain} and say that it is {\it proper} if $\mathrm{dom}(f)\neq\emptyset$ and $f(x)\neq-\infty$ for all $x\in \Hi$. For a  function $f:\Hi\to(-\infty,+\infty]$, the {\it conjugate function} of $f$ is the proper convex lower semicontinuous function $f^*:\Hi\to(-\infty,+\infty]$ defined by $$f^*(x):=\sup_{u\in \Hi}\{\langle x,u\rangle-f(u)\}$$
for all $x\in \Hi$. The {\it subdifferential} of $f$ at $x\in \Hi$ with $f(x)\in \R$ is the set 
$$\partial f(x):=\{x^*\in \Hi:f(y)\geq f(x)+\langle x^*, y-x\rangle \textrm{ for all } x\in \Hi\},$$ 
and we have a convention that $\partial f(x):=\emptyset$ if $f(x)=+\infty$.
Note that the subdifferential of a proper convex lower semicontinuous function is a maximally monotone operator and it holds that
\begin{eqnarray}\label{fptrick}
\varphi_{\partial f}(x,u)\leq f(x)+f^*(u)\indent\textrm{ for all }(x,u)\in \Hi\times \Hi.
\end{eqnarray}

\vspace{2ex}
For $r>0$ and $x\in \Hi$, we denote by $\prox_{rf}(x)$ the {\it proximal point} of parameter $r$ of a proper convex lower semicontinuous function  $f$ at $x$, which is the unique optimal solution of the optimization problem
$$\min_{u\in \Hi} f(u)+\frac{1}{2r}\|u-x\|^2.$$
Note that $\prox_{rf}=J_{r\partial f}$ and it is a single-valued operator.

We say that $f$ is {\it strongly convex} with modulus $\alpha>0$ if $f-\frac{\alpha}{2}\|\cdot\|^2$ is convex. Note that if $f$ is $\alpha-$strongly convex, then its subdifferential is $\alpha-$strongly monotone.

Given a nonempty closed convex set $X\subset \Hi$, its {\it indicator function} is defined as $\delta_X(x)=0$ if $x\in X$ and $+\infty$ otherwise. The {\it support function} of $X$ at a point $x$ is defined by
$$\sigma_X(x):=\sup_{z\in X}\langle x,z\rangle.$$ 
The {\it normal cone} to $X$ at $x$ is $$N_X(x):=\{x^*\in \Hi:\langle x^*,z-x\rangle\leq0 \textrm{ for all }z\in X\},$$
if $x\in X$ and $\emptyset$ otherwise. Note that $x^* \in N_X(x)$ if and only if $\sigma_X(x^*) = \langle x^*,x\rangle$.

Let $\{x_k\}_{k=1}^{\infty}$ be a sequence in $\Hi$ and let $\{\alpha_k\}_{k=1}^{\infty}$ be a sequence of positive real numbers such that $\sum_{k=1}^{\infty}\alpha_k=+\infty$. We define the sequence $\{z_k\}_{k=1}^\infty$ of weighted averages by
\begin{eqnarray}\label{ergodic-eqn}z_k:=\frac{1}{\tau_k}\sum_{n=1}^{k}\alpha_nx_n,\indent \textrm{ where }\tau_k=\sum_{n=1}^k\alpha_n\indent \forall k\geq1.
\end{eqnarray}

The following lemma is a key tool for proving the convergence results.
\begin{lemma}[Opial-Passty]\label{lemma-OP}
	Let $\Hi$ be a real Hilbert space, and let $C\subseteq \Hi$ be a nonempty set. Let $\{x_k\}_{k=1}^\infty$  and $\{z_k\}_{k=1}^\infty$ (defined as in (\ref{ergodic-eqn})) be sequences satisfying the following: 
	
	(i) For every $z\in C$, $\lim_{k\to+\infty}\|x_k-z\|$ exists.
	
	(ii) Every sequential weak cluster point of $\{x_k\}_{k=1}^\infty$ (respectively, $\{z_k\}_{k=1}^\infty$) lies in $C$.
	
	\noindent	Then the sequence $\{x_k\}_{k=1}^\infty$ (respectively, $\{z_k\}_{k=1}^\infty$)  converges weakly to a point in $C$.
\end{lemma}

In order to show the convergence results, we also need the following fact. 
\begin{lemma}
	\label{lemma-acp10}\cite{polyak} Let $\{a_k\}_{k=1}^{\infty}$, $\{b_k\}_{k=1}^{\infty}$, and $\{c_k\}_{k=1}^{\infty}$ be real sequences. Assume that $\{a_k\}_{k=1}^{\infty}$ is bounded from below, and $\{b_k\}_{k=1}^{\infty}$ is nonnegative. If  $$a_{k+1}-a_k+b_k\leq c_k, \forall k\geq1$$ and $\sum_{k=1}^{\infty}c_k<+\infty$, then $\lim_{k\to+\infty}a_k$ exists and $\sum_{k=1}^{\infty}b_k<+\infty$.
\end{lemma}

\section{Convergence Results}

In this section, we are interested in the following generalized forward-backward method with penalty term (in short, GFBP)  for solving MIP. 
\begin{algorithm}[GFBP] \label{algorithm-IFB-P1}
	\textit{Initialization}: Choose the positive real sequences $\{\alpha_k\}_{k=1}^\infty$, $\{\beta_k\}_{k=1}^\infty$, and take arbitrary $x_1\in \Hi$. \\
	\textit{Iterative Step}: For a given current iterate $x_{k}\in \Hi$ ($k\geq 1$), compute $$\psi_{0,k}:=x_{k}-\alpha_kB(x_k)-\alpha_k\beta_kC(x_k).$$ 
	For $i=1,\ldots,m$, compute
	\begin{eqnarray*}\psi_{i,k}:=J_{\alpha_kA_i}(\psi_{i-1,k})
	\end{eqnarray*}
	and define $$x_{k+1}:=\psi_{m,k}.$$
\end{algorithm}

\begin{remark} 
(i) When $m=1$, GFBP reduces to the algorithm  proposed and investigated in \cite{BC14-penalty} for solving (\ref{MIPm=1}).  
On the other hand, if  $B(x)=C(x)= 0$ for all $x\in\Hi$, then  GFBP turns out to be the $m-$fold backward algorithm  proposed in \cite{P79}.

(ii) Recall that when $C(x)=0$ for all $x\in \Hi$, MIP turns out to be the problem of finding $x\in \Hi$ such that
	\begin{eqnarray}\label{MIP-noncone}%
	0\in \sum_{i=1}^mA_i(x)+B(x),
	\end{eqnarray}
	which was investigated by the authors in \cite{RFP13}. They proposed a generalized forward-backward iterative scheme which involves the computation of the cocoercive operator in a forward step and the parallel computation of the resolvents of the $A_i$'s in a subsequent backward step. They proved that the sequence generated by the algorithm converges weakly in the setting of Hilbert space. The authors also showed the applications of the proposed method to  image processing problems. Some more works related to such proposed direction are due to, e.g., \cite{CP12}, \cite{V13}, and \cite{RL15}. It is worth mentioning that our GFBP is different from such methods. In fact,  in order to update the next iteration, those methods allow one to compute the resolvents of $A_i$ at the current iteration for all $i=1,\ldots,m$ and subsequently  combine them to obtain an update to be the next iteration, whereas our GFBP scheme offers one to compute the resolvents of $A_i$ incrementally (see Algorithm \ref{algorithm-IFB-P1}).
 \end{remark}

For the convergence results, we need the following assumption.
\begin{assumption}\label{coco-assump}	
The operators $B$ and $C$ are cocoercive with parameters $\mu$ and $\eta$, respectively. 
\end{assumption}

The following technical lemma will be useful in the convergence analysis of Algorithm \ref{algorithm-IFB-P1}.

\begin{lemma}\label{IFBP-lemma1}
	Let $(u,w)\in \Gr(\sum_{i=1}^mA_i+B+N_{\zer(C)})$, $v_i\in A_i(u)$ for all $i=1,\ldots,m$ and $p\in N_{\zer(C)}(u)$ be such that $w=\sum_{i=1}^mv_i+B(u)+p$. If Assumption \ref{coco-assump} holds, then the following inequality holds for all $k\geq 1$ and  $\ep>0$:
	\begin{eqnarray}\label{IFBP-lemma1-eqn}
	\|x_{k+1}-u\|^2&-&\|x_k-u\|^2+\left(1-\frac{\ep}{1+\ep}\right)\sum_{i=1}^m\|\psi_{i,k}-\psi_{i-1,k}\|^2
	\nonumber\\
	&+&\left(\frac{2\mu}{1+\ep}-\left(2+\frac{\ep}{1+\ep}\right)\alpha_k\beta_k\right)\alpha_k\beta_k\|C(x_k)\|^2\nonumber\\
	&+&\frac{\ep}{1+\ep}\alpha_k\beta_k\<x_k-u,C(x_k)\>\nonumber\\
	&\leq&\left(\left(4+\frac{2\ep}{1+\ep}\right)\alpha_k-2\eta\right)\alpha_k\|B(x_k)-B(u)\|^2\nonumber\\
	&&+\left(4+\frac{2\ep}{1+\ep}\right)\alpha_k^2\|B(u)\|^2+2\alpha_k\<u-x_k,w\>\nonumber\\
	&&+\frac{2(m(m+1)+1)(1+\ep)}{\ep}\alpha_k^2\sum_{i=1}^m\|v_i\|^2\nonumber\\
	&&+\frac{\ep}{1+\ep}\alpha_k\beta_k\left[\sup_{u\in\zer(C)}\varphi_{C}\left(u,\frac{2p}{\frac{\ep}{1+\ep}\beta_k}\right)-\sigma_{\zer(C)}\left(\frac{2p}{\frac{\ep}{1+\ep}\beta_k}\right)\right].\nonumber\\
	\end{eqnarray}
	
\end{lemma}
\begin{proof}
	For all $i=1,\dots,m$ and $k\geq 1$, we know that
	$\psi_{i-1,k}-\psi_{i,k}\in \alpha_kA_i(\psi_{i,k})$.
	Since $v_i\in A_i(u)$, it follows from the monotonicity of $A_i$ that
	\begin{eqnarray}\label{lemma1-proof-1}
	\langle \psi_{i-1,k}-\psi_{i,k}-\alpha_kv_i,\psi_{i,k}-u\rangle\geq0.
	\end{eqnarray}
	Note that for all $i=1,\dots,m$ and $k\geq 1$, we have 
	\begin{eqnarray}\label{lemma-1-proof-1-1}\hspace{-0.3cm}\|\psi_{i,k}-u\|^2-\|\psi_{i-1,k}-u\|^2+\|\psi_{i,k}-\psi_{i-1,k}\|^2=2\langle \psi_{i,k}-\psi_{i-1,k},\psi_{i,k}-u\rangle,
	\end{eqnarray}
	which, together with (\ref{lemma1-proof-1}), imply that
	\begin{eqnarray}\label{lemma1-proof-2}
	\|\psi_{i,k}-u\|^2-\|\psi_{i-1,k}-u\|^2+\|\psi_{i,k}-\psi_{i-1,k}\|^2\leq2\alpha_k\langle v_i,u-\psi_{i,k}\rangle.
	\end{eqnarray}
	Summing up inequalities (\ref{lemma1-proof-2}) for all $i=1,\ldots,m$, we obtain that for every $k\geq1$
	\begin{eqnarray}\label{lemma1-proof-3}
	\hspace{-0.3cm}\|x_{k+1}-u\|^2-\|\psi_{0,k}-u\|^2+\sum_{i=1}^m\|\psi_{i,k}-\psi_{i-1,k}\|^2\leq2\alpha_k\sum_{i=1}^m\langle v_i,u-\psi_{i,k}\rangle.
	\end{eqnarray}
	Note that 
	\begin{eqnarray*}
		\|\psi_{0,k}-u\|^2&=&\|x_k-u\|^2+\alpha_k^2\|B(x_k)+\beta_kC(x_k)\|^2\\
		&&-2\alpha_k\<x_k-u,B(x_k)+\beta_kC(x_k)\>\\
		&\leq&\|x_k-u\|^2+2\alpha_k^2\|B(x_k)\|^2+2\alpha_k^2\beta_k^2\|C(x_k)\|^2\\
		&&-2\alpha_k\<x_k-u,B(x_k)+\beta_kC(x_k)\>	\\
		&\leq&\|x_k-u\|^2+4\alpha_k^2\|B(x_k)-B(u)\|^2+4\alpha_k^2\|B(u)\|^2\\
		&&+2\alpha_k^2\beta_k^2\|C(x_k)\|^2-2\alpha_k\<x_k-u,B(x_k)+\beta_kC(x_k)\>.	
	\end{eqnarray*}	
	It then follows from the inequality (\ref{lemma1-proof-3}) that
	\begin{eqnarray}\label{lemma1-proof-4}
	\|x_{k+1}-u\|^2&-&\|x_k-u\|^2+\sum_{i=1}^m\|\psi_{i,k}-\psi_{i-1,k}\|^2\nonumber\\
	&\leq&4\alpha_k^2\|B(x_k)-B(u)\|^2+4\alpha_k^2\|B(u)\|^2+2\alpha_k^2\beta_k^2\|C(x_k)\|^2\nonumber\\
	&&-2\alpha_k\<x_k-u,B(x_k)+\beta_kC(x_k)\>+2\alpha_k\sum_{i=1}^m\langle v_i,u-\psi_{i,k}\rangle\nonumber\\
	&=&4\alpha_k^2\|B(x_k)-B(u)\|^2+4\alpha_k^2\|B(u)\|^2+2\alpha_k^2\beta_k^2\|C(x_k)\|^2\nonumber\\
	&&+2\alpha_k\beta_k\<u-x_k,C(x_k)\>+2\alpha_k\<u-x_k,B(x_k)+\sum_{i=1}^mv_i\>\nonumber\\	
	&&+2\alpha_k\sum_{i=1}^m\langle v_i,x_k-\psi_{i,k}\rangle,\indent\forall k\geq1.
	\end{eqnarray}
	Since $C$ is $\mu$-cocoercive and $C(u)=0$, we have
	$$2\alpha_k\beta_k\<u-x_k,C(x_k)\>\leq-2\mu\alpha_k\beta_k\|C(x_k)\|^2,\indent\forall k\geq1,$$
	which implies that
	\begin{eqnarray}\label{lemma1-proof-5}
	2\alpha_k\beta_k\<u-x_k,C(x_k)\>&\leq&-\frac{2\mu}{1+\ep}\alpha_k\beta_k\|C(x_k)\|^2\nonumber\\
	&&-\frac{2\ep}{1+\ep}\alpha_k\beta_k\<x_k-u,C(x_k)\>, \indent\forall k\geq1.
	\end{eqnarray}
	
	The $\eta$-cocoercivity of $B$ also yields for all $k\geq1$
	\begin{eqnarray}\label{lemma1-proof-6}
	2\alpha_k\<u-x_k,B(x_k)+\sum_{i=1}^mv_i\>
	&=&2\alpha_k\<u-x_k,B(x_k)-B(u)\>\nonumber\\
	&&+2\alpha_k\<u-x_k,B(u)+\sum_{i=1}^mv_i\>\nonumber\\
	&\leq&-2\eta\alpha_k\|B(x_k)-B(u)\|^2\nonumber\\
	&&+2\alpha_k\<u-x_k,B(u)+\sum_{i=1}^mv_i\>.
	\end{eqnarray}

	Let us consider the last term of (\ref{lemma1-proof-4}). Note that 
	\begin{eqnarray*}
		0&\leq&\frac{1}{m(m+1)}\left(\frac{2(1+\varepsilon)}{\varepsilon}\right)\left\|\frac{\varepsilon}{2(1+\varepsilon)}(x_k-\psi_{i-1,k})-m(m+1)\alpha_kv_i\right\|^2\\
		&=&\frac{\varepsilon}{2m(m+1)(1+\varepsilon)}\|x_k-\psi_{i-1,k}\|^2+\frac{2m(m+1)(1+\varepsilon)}{\varepsilon}\alpha_k^2\|v_i\|^2\\
		&&-2\alpha_k\langle x_k-\psi_{i-1,k},v_i\rangle.
	\end{eqnarray*}
	Summing up this inequality for all $i=1,2,\ldots,m$ gives
	\begin{eqnarray}\label{lemma1-proof-7}
	\hspace{-0.5cm}	2\alpha_k\sum_{i=1}^m\langle x_k-\psi_{i-1,k},v_i\rangle
	&\leq& \frac{\varepsilon}{2m(m+1)(1+\varepsilon)}\sum_{i=1}^m\|x_k-\psi_{i-1,k}\|^2\nonumber\\
	&&+\frac{2m(m+1)(1+\varepsilon)}{\varepsilon}\alpha_k^2\sum_{i=1}^m\|v_i\|^2,\indent\forall k\geq1.
	\end{eqnarray}
	For all $k\geq1$, we claim that
	\begin{eqnarray}\label{lemma1-proof-8}
	\hspace{-0.5cm}		\frac{\varepsilon}{2m(m+1)(1+\varepsilon)}\sum_{i=1}^m\|x_k-\psi_{i-1,k}\|^2&\leq& \frac{\varepsilon}{2(1+\varepsilon)}\|x_k-\psi_{0,k}\|^2\nonumber\\
	&&+\frac{\varepsilon}{2(1+\varepsilon)}\sum_{i=1}^m\|\psi_{i,k}-\psi_{i-1,k}\|^2.
	\end{eqnarray}
	In fact, for all $i=1,\ldots,m$ and $k\geq1$, we have
	\begin{eqnarray*}
		\|x_k-\psi_{i-1,k}\|&\leq& \|x_k-\psi_{0,k}\|+\sum_{j=1}^{i-1}\|\psi_{j,k}-\psi_{j-1,k}\|\\
		&\leq&\|x_k-\psi_{0,k}\|+\sum_{i=1}^m\|\psi_{i,k}-\psi_{i-1,k}\|.
	\end{eqnarray*}
	By using the triangle inequality, we have 
	\begin{eqnarray*}
		\|x_k-\psi_{i-1,k}\|^2
		&\leq&\left(\|x_k-\psi_{0,k}\|+\sum_{i=1}^m\|\psi_{i,k}-\psi_{i-1,k}\|\right)^2\\
		&\leq&(m+1)\left(\|x_k-\psi_{0,k}\|^2+\sum_{i=1}^m\|\psi_{i,k}-\psi_{i-1,k}\|^2\right),
	\end{eqnarray*}
	and therefore,
	\begin{eqnarray*}
		\sum_{i=1}^m\|x_k-\psi_{i-1,k}\|^2
		&\leq&m(m+1)\left(\|x_k-\psi_{0,k}\|^2+\sum_{i=1}^m\|\psi_{i,k}-\psi_{i-1,k}\|^2\right).
	\end{eqnarray*}
	Multiplying both sides of this inequality by 	$\frac{\varepsilon}{2m(m+1)(1+\varepsilon) }$, we obtain  (\ref{lemma1-proof-8}) as desired.
	
	Observe that for all $k\geq1$, we have
	\begin{eqnarray*}\|x_k-\psi_{0,k}\|^2&=&\|\alpha_kB(x_k)-\alpha_k\beta_kC(x_k)\|^2\\
		&\leq&2\alpha_k^2\|B(x_k)\|^2+2\alpha_k^2\beta_k^2\|C(x_k)\|^2\\
		&\leq&4\alpha_k^2\|B(x_k)-B(u)\|^2+4\alpha_k^2\|B(u)\|^2+2\alpha_k^2\beta_k^2\|C(x_k)\|^2.
	\end{eqnarray*}
	This together with (\ref{lemma1-proof-8}) imply that (\ref{lemma1-proof-7}) becomes
	\begin{eqnarray}\label{lemma1-proof-9}
	2\alpha_k\sum_{i=1}^m\langle x_k-\psi_{i-1,k},v_i\rangle
	&\leq& \frac{2\varepsilon}{1+\varepsilon}\alpha_k^2\|B(x_k)-B(u)\|^2+\frac{2\varepsilon}{1+\varepsilon}\alpha_k^2\|B(u)\|^2\nonumber\\
	&&+\frac{\varepsilon}{1+\varepsilon}\alpha_k^2\beta_k^2\|C(x_k)\|^2+\frac{\varepsilon}{2(1+\varepsilon)}\sum_{i=1}^m\|\psi_{i,k}-\psi_{i-1,k}\|^2\nonumber\\
	&&+\frac{2m(m+1)(1+\varepsilon)}{\varepsilon}\alpha_k^2\sum_{i=1}^m\|v_i\|^2. 
	\end{eqnarray}
	Furthermore, we also note that for all $k\geq1$
	\begin{eqnarray*}
		0&\leq&\frac{2(1+\varepsilon)}{\varepsilon}\left\|\frac{\varepsilon}{2(1+\varepsilon)}(\psi_{i-1,k}-\psi_{i,k})-\alpha_kv_i\right\|^2\\
		&=&\frac{\varepsilon}{2(1+\varepsilon)}\|\psi_{i-1,k}-\psi_{i,k}\|^2+\frac{2(1+\varepsilon)\alpha_k^2}{\varepsilon}\|v_i\|^2-2\alpha_k\langle \psi_{i-1,k}-\psi_{i,k}, v_i\rangle,
	\end{eqnarray*}
	which is 
	\begin{eqnarray*}
		2\alpha_k\langle \psi_{i-1,k}-\psi_{i,k}, v_i\rangle
		\leq \frac{\varepsilon}{2(1+\varepsilon)}\|\psi_{i-1,k}-\psi_{i,k}\|^2+\frac{2(1+\varepsilon)\alpha_k^2}{\varepsilon}\|v_i\|^2,
	\end{eqnarray*}
	and so
	\begin{eqnarray}\label{lemma1-proof-10}
	2\alpha_k\sum_{i=1}^m\langle \psi_{i-1,k}-\psi_{i,k}, v_i\rangle
	&\leq& \frac{\varepsilon}{2(1+\varepsilon)}\sum_{i=1}^m\|\psi_{i,k}-\psi_{i-1,k}\|^2\nonumber\\
	&&+\frac{2(1+\varepsilon)}{\varepsilon}\alpha_k^2\sum_{i=1}^m\|v_i\|^2.
	\end{eqnarray}
	By using (\ref{lemma1-proof-9}) and (\ref{lemma1-proof-10}), we have
	\begin{eqnarray}\label{lemma1-proof-10-2}
	2\alpha_k\sum_{i=1}^m\langle v_i, x_k-\psi_{i,k}\rangle 
	&=& 2\alpha_k\sum_{i=1}^m\langle v_i, x_k-\psi_{i-1,k}\rangle+2\alpha_k\sum_{i=1}^m\langle v_i, \psi_{i-1,k}-\psi_{i,k}\rangle\nonumber\\
	&\leq& \frac{2\varepsilon}{1+\varepsilon}\alpha_k^2\|B(x_k)-B(u)\|^2+\frac{2\varepsilon}{1+\varepsilon}\alpha_k^2\|B(u)\|^2\nonumber\\
	&&+\frac{\varepsilon}{1+\varepsilon}\alpha_k^2\beta_k^2\|C(x_k)\|^2+\frac{\varepsilon}{(1+\varepsilon)}\sum_{i=1}^m\|\psi_{i,k}-\psi_{i-1,k}\|^2\nonumber\\
	&&+\frac{2(m(m+1)+1)(1+\varepsilon)}{\varepsilon}\alpha_k^2\sum_{i=1}^m\|v_i\|^2.
	\end{eqnarray}
	Combining (\ref{lemma1-proof-4}), (\ref{lemma1-proof-5}), (\ref{lemma1-proof-6}), and (\ref{lemma1-proof-10-2}) yields 
	\begin{eqnarray}\label{lemma1-proof-11}
	\|x_{k+1}-u\|^2&-&\|x_k-u\|^2+\left(1-\frac{\ep}{1+\ep}\right)\sum_{i=1}^m\|\psi_{i,k}-\psi_{i-1,k}\|^2\nonumber\\
	&+&\left(\frac{2\mu}{1+\ep}-\left(2+\frac{\ep}{1+\ep}\right)\alpha_k\beta_k\right)\alpha_k\beta_k\|C(x_k)\|^2\nonumber\\
	&+&\frac{\ep}{1+\ep}\alpha_k\beta_k\<x_k-u,C(x_k)\>\nonumber\\
	&\leq&\left(\left(4+\frac{2\ep}{1+\ep}\right)\alpha_k-2\eta\right)\alpha_k\|B(x_k)-B(u)\|^2\nonumber\\
	&&+\left(4+\frac{2\ep}{1+\ep}\right)\alpha_k^2\|B(u)\|^2+\frac{2(m(m+1)+1)(1+\ep)}{\ep}\alpha_k^2\sum_{i=1}^m\|v_i\|^2\nonumber\\
	&&+2\alpha_k\<u-x_k,B(u)+\sum_{i=1}^mv_i\>+\frac{\ep}{1+\ep}\alpha_k\beta_k\<x_k-u,C(x_k)\>.
	\end{eqnarray}

	Finally, using the definition of the Fitzpatrick function and the fact that $\sigma_{\zer(C)}\left(\frac{2p}{\frac{\ep}{1+\ep}\beta_k}\right)=\<u,\frac{2p}{\frac{\ep}{1+\ep}\beta_k}\>$ for all $k\geq1$, we have

	\begin{eqnarray*}&&2\alpha_k\<u-x_k,B(u)+\sum_{i=1}^mv_i\>+\frac{\ep}{1+\ep}\alpha_k\beta_k\<u,C(x_k)\>-\frac{\ep}{1+\ep}\alpha_k\beta_k\<x_k,C(x_k)\>	\\
		&=&	2\alpha_k\<u-x_k,-p\>+2\alpha_k\<u-x_k,w\>+\frac{\ep}{1+\ep}\alpha_k\beta_k\<u,C(x_k)\>\\
		&&-\frac{\ep}{1+\ep}\alpha_k\beta_k\<x_k,C(x_k)\>\\
		&=&\frac{\ep}{1+\ep}\alpha_k\beta_k\left[\<x_k,\frac{2p}{\frac{\ep}{1+\ep}\beta_k}\>+\<u,C(x_k)\>-\<x_k,C(x_k)\>-\<u,\frac{2p}{\frac{\ep}{1+\ep}\beta_k}\>\right]\\
		&&+2\alpha_k\<u-x_k,w\>\\
		&\leq&\frac{\ep}{1+\ep}\alpha_k\beta_k\left[\sup_{u\in\zer(C)}\varphi_{C}\left(u,\frac{2p}{\frac{\ep}{1+\ep}\beta_k}\right)-\<u,\frac{2p}{\frac{\ep}{1+\ep}\beta_k}\>\right]\\
		&&+2\alpha_k\<u-x_k,w\>\\
		&=&\frac{\ep}{1+\ep}\alpha_k\beta_k\left[\sup_{u\in\zer(C)}\varphi_{C}\left(u,\frac{2p}{\frac{\ep}{1+\ep}\beta_k}\right)-\sigma_{\zer(C)}\left(\frac{2p}{\frac{\ep}{1+\ep}\beta_k}\right)\right]\\
		&&+2\alpha_k\<u-x_k,w\>,
	\end{eqnarray*}
	which together with (\ref{lemma1-proof-11}) imply the required inequality.
\end{proof}

For the convergence results of this section, the following assumption is required.
\begin{assumption}\label{assumption-IFB-P1}The following statements hold:
	\begin{itemize}
		\item[(H1)] The qualification condition
		$$\zer(C)\cap\bigcap_{i=1}^m\mathrm{int}(\mathrm{dom}(A_i))\neq\emptyset$$
		holds.
		\item[(H2)] For every $p\in \ran(N_{\zer(C)})$, we have $$\sum_{k=1}^\infty\alpha_k\beta_k\left[\sup_{u\in\zer(C)}\varphi_{C}\left(u,\frac{p}{\beta_k}\right)-\sigma_{\zer(C)}\left(\frac{p}{\beta_k}\right)\right]<+\infty.$$	
		\item[(H3)]  The sequences $\{\alpha_k\}_{k=1}^{\infty}$ and $\{\beta_k\}_{k=1}^{\infty}$ satisfy
		$$0<\liminf_{k\to+\infty}\alpha_k\beta_k\leq \limsup_{k\to+\infty}\alpha_k\beta_k <\mu,$$
	where $\mu$ is the cocoercive parameter of the operator $B$.
		\item[(H4)] The sequence $\{\alpha_k\}_{k=1}^{\infty}\in\ell^2\setminus \ell^1$.
	\end{itemize}
\end{assumption}

\begin{remark}Some remarks concerning Assumption \ref{assumption-IFB-P1} are as follows.
	
	(i) Condition (H1) implies that $\sum_{i=1}^mA_i+B+N_{\zer(C)}$ is a maximally monotone operator.
	In fact, since $\bigcap_{i=1}^m\mathrm{int}(\mathrm{dom}(A_i))\neq\emptyset$, the sum $\sum_{i=1}^mA_i$ is maximally monotone (see \cite[Corollary 24.4]{BC11}). The maximal monotonicity of $N_{\zer(C)}$ (see \cite[Proposition 23.39]{BC11}) and the fact that $\bigcap_{i=1}^m\mathrm{int}(\mathrm{dom}(A_i))\subset \mathrm{int}(\mathrm{dom}(\sum_{i=1}^mA_i))$ give us the maximal monotonicity of $\sum_{i=1}^mA_i+N_{\zer(C)}$. Moreover, since $B$ is maximally monotone (see \cite[Example 20.28]{BC11}),  condition (H1) guarantees that $\sum_{i=1}^mA_i+B+N_{\zer(C)}$ is also maximally monotone (see \cite[Corollary 24.4]{BC11}).
	
	(ii) Hypothesis (H2) has  been introduced by Bo\c{t} and Csetnek \cite{BC14-penalty} in order to show the convergence of the proposed iterative scheme (cf. \cite[Hypothesis ($H_{f itz}$)]{BC14-penalty}). They also pointed out that for every $p \in \ran(N_{\zer(C)})$ and any $k\geq1$, one has
	$$\sup_{u\in\zer(C)}\varphi_{C}\left(u,\frac{p}{\beta_k}\right)-\sigma_{\zer(C)}\left(\frac{p}{\beta_k}\right)\geq0.$$
	Some instances of the operator $C$ satisfying  hypothesis (H2) can be found in \cite[Section 5]{BB15}. 
	
	(iii) An example of the sequences $\{\alpha_k\}_{k=1}^{\infty}$ and $\{\beta_k\}_{k=1}^{\infty}$  satisfying conditions (H3) and (H4) is that $\alpha_k=1/k$ and $\beta_k= \xi k$ for all $k\geq1$, where $0<\xi<\mu$.
\end{remark}

 The following theorem is the convergence result for GFBP.
\begin{theorem}\label{theorem-IFB-P1}Let $\{x_k\}_{k=1}^\infty$ be a sequence generated by GFBP and let $\{z_k\}_{k=1}^\infty$ be a sequence of weighted averages as (\ref{ergodic-eqn}). If Assumptions \ref{coco-assump} and \ref{assumption-IFB-P1} hold, then the following statements are true:
	\begin{itemize}
		\item[(i)] For all $u\in\zer(\sum_{i=1}^mA_i+B+N_{\zer(C)})$ we have $\lim_{k\to+\infty}\|x_k-u\|$ exists, and the series $\sum_{k=1}^\infty\sum_{i=1}^m\|\psi_{i,k}-\psi_{i-1,k}\|^2$,  $\sum_{k=1}^\infty\alpha_k\beta_k\|C(x_k)\|^2$, and
		 $\sum_{k=1}^\infty\alpha_k\beta_k\<x_k-u,C(x_k)\>$ are convergent. 
		\item[(ii)] It hols that \begin{eqnarray*}
		\lim_{k\to+\infty}\sum_{i=1}^m\|\psi_{i,k}-\psi_{i-1,k}\|^2&=&\lim_{k\to+\infty}\<x_k-u,C(x_k)\>
		=\lim_{k\to+\infty}\|C(x_k)\|=0.
		\end{eqnarray*}
		\item[(iii)] The sequence $\{z_k\}_{k=1}^\infty$ converges weakly to an element in $\zer(\sum_{i=1}^mA_i+B+N_{\zer(C)})$.
	\end{itemize}
\end{theorem}
\begin{proof} (i) 
	Let $u\in\zer(\sum_{i=1}^mA_i+B+N_{\zer(C)})$. Since $\limsup_{k\to+\infty}\alpha_k\beta_k <\mu$, there exists $k_0\geq 1$  such that $\alpha_k\beta_k <\mu$ for all $k\geq k_0$. Pick 
	$\ep_0\in\left(0,\frac{2(\mu-\limsup_{k\to+\infty}\alpha_k\beta_k)}{3\limsup_{k\to+\infty}\alpha_k\beta_k}\right).$ It follows that  
	$\alpha_k\beta_k<\frac{2\mu}{2+3\ep_0}$ for all $k\geq k_0$. This allows us to choose $M>0$ such that $\alpha_k\beta_k\leq M<\frac{2\mu}{2+3\ep_0}=\frac{2\mu}{(1+\ep_0)\left(2+\frac{\ep_0}{1+\ep_0}\right)}$ for all $k\geq k_0$.
	Furthermore, since $\ep_0>0$ and $\alpha_k\to 0$, there exists $k_1\geq1$ such that 
	$$\left(4+\frac{2\ep_0}{1+\ep_0}\right)\alpha_k-2\eta<0,\indent \forall k\geq k_1.$$ 
	Taking $w=0$ in Lemma \ref{IFBP-lemma1}, for every $k\geq \overline{k}:=\max\{k_0,k_1\}$, we obtain 
	\begin{eqnarray}\label{IFBP-thm1-proof-eqn1}
	\hspace{-0.4cm}	\|x_{k+1}-u\|^2&-&\|x_k-u\|^2+\left(1-\frac{\ep_0}{1+\ep_0}\right)\sum_{i=1}^m\|\psi_{i,k}-\psi_{i-1,k}\|^2
	\nonumber\\
&+&\left(\frac{2\mu}{1+\ep_0}-\left(2+\frac{\ep_0}{1+\ep_0}\right)M\right)\alpha_k\beta_k\|C(x_k)\|^2\nonumber\\
	&&+\frac{\ep_0}{1+\ep_0}\alpha_k\beta_k\<x_k-u,C(x_k)\>\nonumber\\
	&\leq&\frac{2(m(m+1)+1)(1+\ep_0)}{\ep_0}\alpha_k^2\sum_{i=1}^m\|v_i\|^2\nonumber\\
	&&+\frac{\ep_0}{1+\ep_0}\alpha_k\beta_k\left[\sup_{u\in\zer(C)}\varphi_{C}\left(u,\frac{2p}{\frac{\ep_0}{1+\ep_0}\beta_k}\right)-\sigma_{\zer(C)}\left(\frac{2p}{\frac{\ep_0}{1+\ep_0}\beta_k}\right)\right].\nonumber\\	
	\end{eqnarray}
	Since the right-hand side is summable and the term $\<x_k-u,C(x_k)\>$ is nonnegative for all $k\geq1$, the conclusion in (i) follows from Lemma \ref{lemma-acp10} and Assumption \ref{assumption-IFB-P1} ((H2) and (H4)).

	(ii) Note that $\lim_{k\to+\infty}\sum_{i=1}^m\|\psi_{i,k}-\psi_{i-1,k}\|^2=0$. Since $\liminf_{k\to+\infty}\alpha_k\beta_k>0$, we have $\lim_{k\to+\infty}\|C(x_k)\|=\lim_{k\to+\infty}\<x_k-u,C(x_k)\>=0$. 

	(iii) Let $z$ be a weak cluster point of $\{x_k\}_{k=1}^{\infty}$ and $\{x_{k_j}\}_{j=1}^{\infty}$ be a subsequence of $\{x_k\}_{k=1}^{\infty}$ such that $x_{k_j}\rightharpoonup z$. Since $\sum_{i=1}^mA_i+B+N_{\zer(C)}$ is a maximally monotone operator, in order to show that $z\in\zer\left(\sum_{i=1}^mA_i+B+N_{\zer(C)}\right)$, we will show that $\<u-z,w\>\geq0$ for all $(u,w)\in\Gr\left(\sum_{i=1}^mA_i+B+N_{\zer(C)}\right)$. 
	Now let $(u,w)\in\Gr\left(\sum_{i=1}^mA_i+B+N_{\zer(C)}\right)$ be such that $w=\sum_{i=1}^mv_i+B(u)+p$, where $v_i\in A_i(u)$ for all $i=1,\ldots,m$ and $p\in N_{\zer(C)}(u)$. From Lemma \ref{IFBP-lemma1}, for every $k\geq \overline{k}$, we have 
	\begin{eqnarray*}
		\|x_{k+1}-u\|^2-\|x_k-u\|^2
		&\leq&2\alpha_k\<u-x_k,w\>+\frac{2(m(m+1)+1)(1+\ep_0)}{\ep_0}\alpha_k^2\sum_{i=1}^m\|v_i\|^2\nonumber\\
		&&+\frac{\ep_0}{1+\ep_0}\alpha_k\beta_k\left[\sup_{u\in\zer(C)}\varphi_{C}\left(u,\frac{2p}{\frac{\ep_0}{1+\ep_0}\beta_k}\right)-\sigma_{\zer(C)}\left(\frac{2p}{\frac{\ep_0}{1+\ep_0}\beta_k}\right)\right].
	\end{eqnarray*}
	Summing up this inequality for all $k=\overline{k}+1,\ldots,k_j$, we obtain
	\begin{eqnarray*}
		\|x_{k_j}-u\|^2-\|x_{\overline{k}}-u\|^2&\leq&2\<\sum_{k=\overline{k}}^{k_j}\alpha_ku-\sum_{k=\overline{k}}^{k_j}\alpha_kx_k,w\>+L_1\\
		&=&2\<\sum_{k=1}^{k_j}\alpha_ku-\sum_{k=1}^{\overline{k}}\alpha_ku-\sum_{k=1}^{k_j}\alpha_kx_k+\sum_{k=1}^{\overline{k}}\alpha_ku,w\>+L_1,
	\end{eqnarray*}
	where \begin{eqnarray*}L_1&:=&\frac{2(m(m+1)+1)(1+\ep_0)}{\ep_0}\sum_{i=1}^m\|v_i\|^2\sum_{k=\overline{k}}^{k_j}\alpha_k^2\\
		&&+\frac{\ep_0}{1+\ep_0}\sum_{k=\overline{k}}^{k_j}\alpha_k\beta_k\left[\sup_{u\in\zer(C)}\varphi_{C}\left(u,\frac{2p}{\frac{\ep_0}{1+\ep_0}\beta_k}\right)-\sigma_{\zer(C)}\left(\frac{2p}{\frac{\ep_0}{1+\ep_0}\beta_k}\right)\right].
	\end{eqnarray*}
	Discarding the nonnegative term $\|x_{k_j+1}-u\|^2$ and dividing by $2\tau_{k_j}$, we deduce that
	\begin{eqnarray*}
		-\frac{\|x_{\overline{k}}-u\|^2}{2\tau_{k_j}}&\leq&\<u-z_{k_j},w\>+\frac{L_2}{2\tau_{k_j}},
	\end{eqnarray*}
	where $L_2:=L_1+2\<-\sum_{k=1}^{\overline{k}}\alpha_ku+\sum_{k=1}^{\overline{k}}\alpha_ku,w\>$, which is a finite real number. Hence, by passing the limit as $j\to+\infty$ (so that $\lim_{j\to+\infty}\tau_{k_j}=+\infty$), we have 
	$$\<u-z,w\>\geq0.$$
	Since $(u,w)\in\Gr\left(\sum_{i=1}^mA_i+B+N_{\zer(C)}\right)$ is arbitrary, we obtain that $$z\in\zer(\sum_{i=1}^mA_i+B+N_{\zer(C)}).$$ Thanks to Lemma \ref{lemma-OP}, we conclude that the sequence $\{z_k\}_{k=1}^\infty$ converges weakly to an element in $\zer(\sum_{i=1}^mA_i+B+N_{\zer(C)})$.
\end{proof}

 \begin{remark} In case $m=1$, Theorem \ref{theorem-IFB-P1} coincides with Theorem 13 in \cite{BC14-penalty}.   By taking $B(x)=C(x)=0$ for all $x\in \Hi$, Theorem \ref{theorem-IFB-P1} coincides with Theorem 3 of Passty\cite{P79}. 
 \end{remark}

	If only one of the operators $A_i, i=1\ldots,m$, is strongly monotone, then we can prove the strong convergence of the sequence $\{x_k\}_{k=1}^{\infty}$ to the unique zero of MIP as we illustrate in the following theorem. In this case, we assume without loss of generality that the $m^{\textrm{th}}$ operator is strongly monotone.

	\begin{theorem}\label{theorem-IFB-P1-strongly}
		Let $\{x_k\}_{k=1}^\infty$ be a sequence generated by GFBP. If Assumptions \ref{coco-assump} and \ref{assumption-IFB-P1} hold and the operator $A_m$ is $\gamma$-strongly monotone with $\gamma>0$, then the sequence $\{x_k\}_{k=1}^{\infty}$ converges strongly to the unique zero of the operator $\sum_{i=1}^{m}A_i+B+N_{\zer(C)}$. 
	\end{theorem}
	\begin{proof}Let $u$ be the unique element in $\zer\left(\sum_{i=1}^mA_i+B+N_{\zer(C)}\right)$. Then there exist $v_i\in A_i(u)$ for all $i=1,\ldots,m$ and $p\in N_{\zer(C)}(u)$ such that $0=\sum_{i=1}^mv_i+B(u)+p$. Since the arguments in the proof of Lemma \ref{IFBP-lemma1} (inequality (\ref{lemma1-proof-2})) hold for all $i=1,\ldots,m-1$, we have
		\begin{eqnarray}\label{IFBP-thm2-proof-1}
		\|\psi_{i,k}-u\|^2-\|\psi_{i-1,k}-u\|^2+\|\psi_{i,k}-\psi_{i-1,k}\|^2\leq2\alpha_k\<v_i,u-\psi_{i,k}\>
		\end{eqnarray}
		for all $i=1,\ldots,m-1$  and $k\geq1$.	Furthermore, the strong monotonicity of the operator $A_m$ gives
		$$\<\psi_{m-1,k}-\psi_{m,k}-\alpha_kv_m,\psi_{m,k}-u\>\geq\alpha_k\gamma\|\psi_{m,k}-u\|^2,$$
		and so
		\begin{eqnarray}\label{IFPB-thm2-proof-2}
	\hspace{-0.3cm}	\<\psi_{m-1,k}-\psi_{m,k},\psi_{m,k}-u\>\geq\alpha_k\gamma\|x_{k+1}-u\|^2+\alpha_k\<v_m,\psi_{m,k}-u\>,\forall k\geq1.
		\end{eqnarray}
		This together with (\ref{lemma-1-proof-1-1}) imply that for all $k\geq1$
		\begin{eqnarray}\label{IFBP-thm2-proof-3}
		\hspace{-0.4cm}	\|\psi_{m,k}-u\|^2-\|\psi_{m-1,k}-u\|^2+\|\psi_{m,k}-\psi_{m-1,k}\|^2&\leq&-2\alpha_k\gamma\|x_{k+1}-u\|^2\nonumber\\
		&&+2\alpha_k\<v_m,u-\psi_{m,k}\>.
		\end{eqnarray}
		Summing up  inequality (\ref{IFBP-thm2-proof-1}) for $i=1,\ldots,m-1$ and the inequality (\ref{IFBP-thm2-proof-3}), we obtain
		\begin{eqnarray}\label{IFBP-thm2-proof-4}
		\hspace{-0.3cm}		2\alpha_k\gamma\|x_{k+1}-u\|^2+\|x_{k+1}-u\|^2-\|\psi_{0,k}-u\|^2&+&\sum_{i=1}^m\|\psi_{i,k}-\psi_{i-1,k}\|^2\nonumber\\
		&&\leq2\alpha_k\sum_{i=1}^m\<v_i,u-\psi_{i,k}\>.
		\end{eqnarray}
		Following the lines of the proof of Lemma \ref{IFBP-lemma1} (with $w=0$), for all $k\geq1$ and $\ep>0$, we obtain 
		\begin{eqnarray*}
			2\alpha_k\gamma\|x_{k+1}-u\|^2&+&\|x_{k+1}-u\|^2-\|x_k-u\|^2+\left(1-\frac{\ep}{1+\ep}\right)\sum_{i=1}^m\|\psi_{i,k}-\psi_{i-1,k}\|^2
			\nonumber\\
			&+&\left(\frac{2\mu}{1+\ep}-\left(2+\frac{\ep}{1+\ep}\right)\alpha_k\beta_k\right)\alpha_k\beta_k\|C(x_k)\|^2\nonumber\\
			&+&\frac{\ep}{1+\ep}\alpha_k\beta_k\<x_k-u,C(x_k)\>\nonumber\\
			&\leq&\left(\left(4+\frac{2\ep}{1+\ep}\right)\alpha_k-2\eta\right)\alpha_k\|B(x_k)-B(u)\|^2\nonumber\\
			&&+\left(4+\frac{2\ep}{1+\ep}\right)\alpha_k^2\|B(u)\|^2+\frac{2(m(m+1)+1)(1+\ep)}{\ep}\alpha_k^2\sum_{i=1}^m\|v_i\|^2\nonumber\\
			&&+\frac{\ep}{1+\ep}\alpha_k\beta_k\left[\sup_{u\in\zer(C)}\varphi_{C}\left(u,\frac{2p}{\frac{\ep}{1+\ep}\beta_k}\right)-\sigma_{\zer(C)}\left(\frac{2p}{\frac{\ep}{1+\ep}\beta_k}\right)\right].
		\end{eqnarray*}
		Using the line of the proof of Theorem \ref{theorem-IFB-P1} (i),  for every $k\geq \overline{k}$, we have 
		\begin{eqnarray}\label{IFBP-thm2-proof-5}
		2\alpha_k\gamma\|x_{k+1}-u\|^2&+&\|x_{k+1}-u\|^2-\|x_k-u\|^2\nonumber\\
		&\leq&\frac{2(m(m+1)+1)(1+\ep_0)}{\ep_0}\alpha_k^2\sum_{i=1}^m\|v_i\|^2\nonumber\\
		&&+\frac{\ep_0}{1+\ep_0}\alpha_k\beta_k\left[\sup_{u\in\zer(C)}\varphi_{C}\left(u,\frac{2p}{\frac{\ep_0}{1+\ep_0}\beta_k}\right)-\sigma_{\zer(C)}\left(\frac{2p}{\frac{\ep_0}{1+\ep_0}\beta_k}\right)\right],\nonumber\\
		\end{eqnarray}
		and so
		\begin{eqnarray*}
			2\gamma\sum_{k=\overline{k}}^{+\infty}\alpha_k\|x_{k+1}-u\|^2
			&\leq&\|x_{\overline{k}}-u\|^2+\frac{2(m(m+1)+1)(1+\ep_0)}{\ep_0}\sum_{i=1}^m\|v_i\|^2\sum_{k=\overline{k}}^{+\infty}\alpha_k^2\nonumber\\
			&&+\frac{\ep_0}{1+\ep_0}\sum_{k=\overline{k}}^{+\infty}\alpha_k\beta_k\left[\sup_{u\in\zer(C)}\varphi_{C}\left(u,\frac{2p}{\frac{\ep_0}{1+\ep_0}\beta_k}\right)-\sigma_{\zer(C)}\left(\frac{2p}{\frac{\ep_0}{1+\ep_0}\beta_k}\right)\right].
		\end{eqnarray*}
		Since $\sum_{k=1}^{+\infty}\alpha_k=+\infty$ and $\lim_{k\to+\infty}\|x_k-u\|$ exists (by (\ref{IFBP-thm2-proof-5}) and Lemma \ref{lemma-acp10}), we conclude that $\lim_{k\to+\infty}\|x_k-u\|=0$.
	\end{proof}

	\section{Hierarchical Minimization Problem}
	
	In this section we show that the iterative scheme proposed in the previous section allows for solving of hierarchical minimization problem. The problem under investigation is of the form
	\begin{eqnarray}\label{opt-plus}%
	\begin{array}{ll}
	\textrm{minimize}\indent \sum_{i=1}^mf_i(x)+h(x)\\
	\textrm{subject to}\indent x\in\argmin g,
	\end{array}%
	\end{eqnarray}
	where, for all $i=1,\ldots,m$, $f_i:\Hi\to(-\infty,+\infty]$ is a proper convex lower semicontinuous objective function, $h:\Hi\to \R$ is a convex (Fr\'echet) differentiable objective function  and $\argmin g$ is the set of minima of a  convex (Fr\'echet) differentiable function $g:\Hi\to \R$ which we will assume that it is nonempty. 
	We assume that the gradient $\nabla h$ and $\nabla g$ are Lipschitz continuous operators with constants $L_h$ and $L_g$, respectively. Furthermore,  we may assume without loss of generality that $\min g=0$. We denote the solution set of the problem (\ref{opt-plus}) by $\mathcal{S}$ and assume that it is a nonempty set.

Since the functions $f_i:\Hi\to(-\infty,+\infty], i=1,\ldots,m$ are  proper convex lower semicontinuous, we know that the subdifferentials  $\partial f_i, i=1,\ldots,m$ are maximally monotone.  Moreover, since the functions $h$ and $g$ are convex differentiable, the Ballion-Haddad\cite{BH77} theorem implies that $\nabla h$ is $\frac{1}{L_h}$-cocoercive and $\nabla g$ is $\frac{1}{Lg}$-cocoercive. Setting $A_i:=\partial f_i$ for all $i=1,\ldots,m$, $B:=\nabla h$ and $C:=\nabla g$, the hierarchical minimization problem (\ref{opt-plus}) is nothing else but a special case of MIP.

	Hence, in order to solve problem (\ref{opt-plus}) we consider the following algorithm.
	\begin{algorithm}\label{algorithm-IFB-Plus}
		\textit{Initialization}: Choose  positive real sequences $\{\alpha_k\}_{k=1}^\infty$, $\{\beta_k\}_{k=1}^\infty$ and take arbitrary $x_1\in \Hi$. \\
		\textit{Iterative Step}: For a given current iterate $x_{k}\in \Hi$ ($k\geq 1$), compute $$\psi_{0,k}:=x_{k}-\alpha_k\nabla h(x_k)-\alpha_k\beta_k\nabla g(x_k).$$ 
		For $i=1,\ldots,m$, compute
		\begin{eqnarray*}\psi_{i,k}:=\mathrm{prox}_{\alpha_kf_i}(\psi_{i-1,k})
		\end{eqnarray*}
		and define $$x_{k+1}:=\psi_{m,k}.$$
	\end{algorithm}

	To obtain the convergence of the sequence generated by Algorithm \ref{algorithm-IFB-Plus}, we need to assume the following  assumption. 

	\begin{assumption}\label{assumption-IFB-P2}The following statements hold:
	\begin{itemize}
		\item[(S1)] The qualification condition
		$$\argmin g\cap\bigcap_{i=1}^m\mathrm{int}(\mathrm{dom}(\partial f_i))\neq\emptyset$$
		holds.
		\item[(S2)] For every $p\in \ran(N_{\argmin g})$, we have $$\sum_{k=1}^\infty\alpha_k\beta_k\left[g^*\left(\frac{p}{\beta_k}\right)-\sigma_{\argmin g}\left(\frac{p}{\beta_k}\right)\right]<+\infty.$$	
		\item[(S3)] The sequences $\{\alpha_k\}_{k=1}^\infty$ and $\{\beta_k\}_{k=1}^\infty$ satisfy
		$$0<\liminf_{k\to+\infty}\alpha_k\beta_k\leq \limsup_{k\to+\infty}\alpha_k\beta_k <\frac{1}{ L_h}.$$
		\item[(S4)] The sequence $\{\alpha_k\}_{k=1}^{\infty}\in\ell^2\setminus \ell^1$.
	\end{itemize}
\end{assumption}

	One can see that conditions (S1)-(S4) imply hypotheses (H1)-(H4) in Assumption \ref{assumption-IFB-P1}. Note that  the implications (S1)$\Rightarrow$(H1), (S3)$\Rightarrow$(H3)
	and (S4)$\Rightarrow$(H4) are obvious and so it suffices to consider the implication (S2)$\Rightarrow$(H2). In fact, according to the relation in (\ref{fptrick}), one has 
	$$\varphi_{\nabla g}(u,\frac{p}{\beta_k})\leq g(u)+g^*\left(\frac{p}{\beta_k}\right)=g^*\left(\frac{p}{\beta_k}\right),$$
	for all $u\in\argmin g$, which implies that 
	$$\sup_{u\in\argmin g}\varphi_{\nabla g}(u,\frac{p}{\beta_k})\leq g^*\left(\frac{p}{\beta_k}\right).$$
	Thus, we obtain that (S2)$\Rightarrow$(H2). Therefore, the following corollary is a direct consequence of Theorem \ref{theorem-IFB-P1}.

	\begin{corollary}\label{cor-IFB-P1}Let $\{x_k\}_{k=1}^\infty$ be a sequence generated by Algorithm \ref{algorithm-IFB-Plus} and let $\{z_k\}_{k=1}^\infty$ be a sequence of weighted averages defined as in (\ref{ergodic-eqn}). If Assumption \ref{assumption-IFB-P2} holds, then the sequence $\{z_k\}_{k=1}^\infty$ converges weakly to an element in $\mathcal{S}$.
	\end{corollary}

If we assume that the $m^{\textrm{th}}$ function $f_m$ is strongly convex, then its subdifferential $\partial f_m$ is strongly monotone. Using this fact, Theorem \ref{theorem-IFB-P1-strongly} also implies the following corollary.

	\begin{corollary}\label{cor-IFB-P1-strongly}
		Let $\{x_k\}_{k=1}^\infty$ be a sequence generated by Algorithm \ref{algorithm-IFB-Plus}. If Assumption \ref{assumption-IFB-P2}  holds  and the function $f_m$ is strongly convex, then the sequence $\{x_k\}_{k=1}^{\infty}$ converges strongly to the unique element in $\mathcal{S}$. 
	\end{corollary}

	\begin{remark}
		In addition, there are other works related to the proposed algorithms and convergence results in the literature. Indeed, Attouch, Czarnecki, and Peypouquet \cite{ACP11-2}  investigated  the monotone inclusion problem of the form
		\begin{eqnarray}\label{MIP-ACP}%
		0\in A(x)+N_{\argmin g}(x),
		\end{eqnarray}
		where $A:\Hi\rightrightarrows \Hi$ is a maximally monotone operator and $\argmin g$ is the set of minima of a proper convex lower semicontinuous function $g:\Hi\to (-\infty,+\infty]$.  The iterative schemes for finding a solution of the problem and  some ergodic convergent results were discussed. Note that in their paper, the authors also considered the problem 
		\begin{eqnarray}\label{MIP-ACP-2}%
		0\in \sum_{i=1}^mA_i(x)+N_{\argmin g}(x),
		\end{eqnarray}
		where, for all $i=1,\ldots,m$, $A_i:\Hi\rightrightarrows \Hi$ is a maximally monotone operator. Of course, the results will certainly be special cases of our main results. 
	Furthermore, the authors in \cite{BC16} considered the modifications of the iterative schemes for solving  (\ref{MIP-ACP}) by employing inertial effects. The main feature of this method is that the next iterate is defined by means of the last two iterates. It is also well known that iterative methods with inertial effects may lead to a considerable improvement of the convergence behavior of the method. We refer the reader to \cite{AP16,BC16-num,BCN17-2,CCMY15,CMY15,OCBP14} and the references therein for more insight into this research topic. Of course, one way to address this research direction is to consider the inertial effects of the proposed algorithms and to analyze their convergence results.
	\end{remark}
	
\section{Numerical Examples}
In this section, we present   the behavior of the algorithm introduced in this paper in the context of two numerical examples on constrained elastic net problems and generalized Heron problems. All the experiments were performed under MATLAB 9.1 (R2016b) running on a MacBook Air 13-inch, Early 2015 with a 1.6GHz Intel Core i5 processor and 4GB 1600MHz DDR3 memory.

\subsection{Constrained Elastic Net}

In this subsection, we consider the constrained linear elastic net problems. First mentioned in the seminal work of Zou and Hastie\cite{ZH05}, the elastic net  has been affirmed that it has the outperformance than the classical ridge regression\cite{HK70} related to minimizing of the residual sum of squares adding the penalty term of $\ell_2-$norm. Furthermore, it also outperforms the Lasso regression\cite{T96} related to minimizing of the residual sum of squares imposing an  $\ell_1-$norm penalization.

Suppose that the data set has $m$ observations with $n$ predictors. Let $\mathbf{b}=(b_1,\ldots,b_m)\in\R^m$ be the response vector and $\mathbf{A}=[\mathbf{a}_1|\cdots|\mathbf{a}_m]^\top\in \R^{m\times n}$ be the design matrix of predictors  $\mathbf{a}_i=(a_{1i},\ldots,a_{ni})\in\R^n$  for all $i=1\ldots,m$. The elastic net problem is to find a solution of
\begin{eqnarray*}%
	\begin{array}{ll}
		\textrm{minimize}\indent \frac{1}{2}\|\mathbf{A}x-\mathbf{b}\|_2^2+\gamma\|x\|_1+(1-\gamma)\|x\|_2^2\\
		\textrm{subject to}\indent x\in\R^n,
	\end{array}%
\end{eqnarray*}
where $\gamma\in [0,1]$ is the elastic net parameter. In this case, we say that the problem has the size $(m,n)$. We note that in some practical situation, for instance, Huang et al.\cite{HGYH13} considered the protein inference problem of selecting a proper subset of candidate proteins that best explain the observed peptides. To achieve efficient computation, the authors formulated the considered protein inference problem as a constrained Lasso regression problem. This brings us to the following constrained elastic net problem:
\begin{eqnarray}\label{Con-elastic-net-classical}%
\begin{array}{ll}
\textrm{minimize}\indent F(x):=\frac{1}{2}\|\mathbf{A}x-\mathbf{b}\|_2^2+\gamma\|x\|_1+(1-\gamma)\|x\|_2^2\\
\textrm{subject to}\indent x\in \argmin \frac{1}{2}\dist^2(\cdot,[0,1]^n),
\end{array}%
\end{eqnarray}
 where the distance function is given by
 $$\dist(x,C):=\inf_{c\in C}\|x-c\|_2.$$
 Note that the distance function $\dist(\cdot,C)$ is nonnegative uniformly continuous\cite[Example 1.47]{BC11} and  convex\cite[Coroolary 12.12]{BC11} whenever  $C$ is a convex set. Furthermore, if $C$ is a nonempty closed and convex set, then the function $\frac{1}{2}\dist^2(\cdot,C)$ is convex (Fr\'echet) differentiable\cite[Corollary 12.30]{BC11}.

One can see that  problem (\ref{Con-elastic-net-classical}) is a minimization problem  of the form (\ref{opt-plus}) when setting  $f_1(x)=\frac{1}{2}\|\mathbf{A}x-\mathbf{b}\|_2^2$,  $f_2(x)=\gamma\|x\|_1$, $f_3(x)=(1-\gamma)\|x\|_2^2$,  and $g(x)=\frac{1}{2}\dist^2(x,[0,1]^n)$ for all $x\in \R^n$.

On the other hand, we can consider the constrained elastic net in the sense of the splitting technique as
\begin{eqnarray}\label{Con-elastic-net}%
\begin{array}{ll}
\textrm{minimize}\indent F(x):=\frac{1}{2}\sum_{i=1}^m(\mathbf{a}_i^\top x-b_i)^2+\gamma\|x\|_1+(1-\gamma)\|x\|_2^2\\
\textrm{subject to}\indent x\in \argmin \frac{1}{2}\dist^2(\cdot,[0,1]^n).
\end{array}%
\end{eqnarray}
We observe that  problem (\ref{Con-elastic-net}) can be written in the form (\ref{opt-plus}) when setting $f_i(x)=\frac{1}{2}(\mathbf{a}_i^\top x-b_i)^2$ for all $i=1,\ldots,m$, $f_{m+1}(x)=\gamma\|x\|_1$, $f_{m+2}=(1-\gamma)\|x\|_2^2$, and $g(x)=\frac{1}{2}\dist^2(x,[0,1]^n)$ for all $x\in\R^n$. Thanks to \cite[Example 6.2]{BB15}, it is guaranteed that the assumption (S2) (in general, (H2)) is satisfied.

\begin{table}
	\centering
	\setlength{\tabcolsep}{12pt}
	\caption{\label{random-tb}Comparisons of number of iterations and algorithm runtime between non-splitting and splitting schemes for different sizes of  matrix $\mathbf{A}$.}
	\label{compare-with-A}
	\begin{tabular}{@{}l   r r r r r r @{}}
			\hline\noalign{\smallskip}
		Setting $\rightarrow$  &  \multicolumn{2}{c}{ Non-splitting (\ref{Con-elastic-net-classical}) } & \multicolumn{2}{c}{ Splitting (\ref{Con-elastic-net}) } \\  
		\cmidrule(l){2-3}  \cmidrule(l){4-5} 
		$(m,n)$ $\downarrow$  	  &  \#(Iters) & Time  &  \#(Iters) &  Time\\ 
			\noalign{\smallskip}\hline\noalign{\smallskip}
		$(20,1000)$                	      & 852	            & 111.85  &  512	           & 1.53	 \\
		$(50,1000)$                	      &  301            & 35.95  & 480              &  	3.74  \\
		$(100,1000)$                     & 264             & 32.74  & 492               & 6.53  \\
		$(200,1000)$                     & 226             & 24.24  & 505               & 11.02 \\
		$(300,1000)$                     & 184             & 20.98  & 503               & 15.65 \\ \midrule
		$(20,2000)$	                      & 862             & 539.47 & 523             & 2.11  \\
		$(50,2000)$                       & 703             & 447.15  & 507             & 4.13  \\
		$(100,2000)$                      & 616             & 338.25  & 513             & 6.38 \\
		$(200,2000)$                      & 373            & 163.82   & 521             & 11.77 \\
		$(300,2000)$                      & 344            & 168.69  & 522             & 19.28 \\ \midrule
		$(20,5000)$                        & 1446           & 7486.05 & 578           & 4.70 \\
		$(50,5000)$                        & 1244           & 5834.81  & 509           & 7.18 \\
		$(100,5000)$                       & 1216           & 5754.33  & 541            & 13.96 \\
		$(200,5000)$                       & 886           & 4200.43  & 558            & 29.22 \\
		$(300,5000)$                       & 724            & 3467.67  & 557            & 47.24 \\ \midrule
		$(20,8000)$                        & 2697          & 48759.21 & 666           & 7.50 \\
		$(50,8000)$                        & 2263          & 45112.77  & 567           & 11.69 \\
		$(100,8000)$                      & 2109           & 40141.74  & 632           &  32.67 \\
		$(200,8000)$                     & 1617            & 31036.15 & 608            & 49.21  \\
		$(300,8000)$                      & 1553          & 34217.33 & 613             & 112.09	\\ 
			\noalign{\smallskip}\hline
	\end{tabular}
\end{table}

In this experiment, the MATLAB calculations are performed by GFBP (Algorithm \ref{algorithm-IFB-Plus}) with the sequences  $\alpha_k=1/k$ and $\beta_k=0.9k$ for all $k\geq1$ and the parameter $\gamma=0.5$.  We use the design matrix $\mathbf{A}$ in $\R^{m\times n}$ generated by the gene expression data set $\texttt{14\_Tumors}$ from {\ttfamily\color{red}http://www.gems-system.org/}.
We generate the vector $\mathbf{b}\in \R^m$ corresponding to $\mathbf{A}$ by  the linear model $\mathbf{b} = \mathbf{A}x_0+\varepsilon$, where $\varepsilon\sim\mathcal{N}(0,\|\mathbf{A}x_0\|^2)$ and $50\%$ percent of the components of vector $x_0$ are nonzero components with normally distributed random generation. We perform GFBP (Algorithm \ref{algorithm-IFB-Plus}) and obtain the number of iterations ($k$) and elapsed time (seconds) by  using the relative change
$$\max\left\{\frac{|F(x_k)-F(x_{k-1})|}{|F(x_{k-1})|},\frac{|g(x_k)-g(x_{k-1})|}{|g(x_{k-1})|}\right\}\leq\epsilon,$$
where $\epsilon$ is an optimality tolerance; in this example,  we use the optimality tolerance $\epsilon=10^{-5}$.    

In Table \ref{compare-with-A}, we compare the numerical results of GFBP for splitting and non-splitting cases for various problem sizes. We see that the splitting case behaves significantly better than the non-splitting counterpart for all problem sizes $(m,n)$. In all cases, the computational time for splitting case is 1.3-6500 times less than that of the non-splitting case. We observe that for the splitting case with the same number of predictors $(n)$, the larger number of observations $(m)$ requires more computational time than the small one. However, the non-splitting case is in another direction: for the same number of predictors $(n)$, it can be seen that the larger number of observations $(m)$ needs less CPU time than the small one for almost all cases. The best performance of the splitting case with respect to the non-splitting counterpart is observed for the problem where $m$ and $n$ are too different. In fact, for the case $(20,8000)$, we see that the CPU time for the splitting case is  $6500$ times less than that of the non-splitting case.

Next, we consider the problem in the sense of splitting when $m= 7, 8, 9, 10$ and $n=2^m$. The design matrix is generated by $\mathbf{A}_{i,j}=\frac{1}{(i+j-1)}$ for all $i=1,\ldots,m, j=1,\ldots,n$ and the response vector $b_i$ is generated by $b_i=-\sum_{j=1}^n\mathbf{A}_{i,j}$ for all $i=1,\ldots,m$. Table \ref{Hilber-m-3m}  shows the behaviors of iterations for the problems with parameters $\gamma=0.1, 0.3, 0.5, 0.7,$ and $0.9$, respectively. We use the optimality tolerance $\epsilon=10^{-6}$ to obtain the number of iteration ($k$) and elapsed time (seconds). In Table \ref{Hilber-m-3m},  we see that the parameter $\gamma=0.9$ uses the smallest number of iterations for all size of problem. Furthermore, we can observe that for each parameter $\gamma$, the relation between  problem sizes growths and number of iterations seems to be linear. Table \ref{Hilber-m-3m}  also shows algorithm runtime when the optimality tolerance $\epsilon=10^{-6}$ is reached. From Table \ref{Hilber-m-3m},  we see that the elapsed time are  not almost in the same direction with the number of iterations. In fact, even if the parameter $\gamma=0.9$ use the smallest algorithm runtime in the cases $m= 7, 9,$ and $10$, the parameter $\gamma=0.7$ uses the least algorithm runtime in the case $m= 8$. Based on the provided data, one can say that  in this example, the parameter $\gamma$ has an impact on the performance of the algorithm.

\begin{table}
	\centering
	\caption{Comparisons of number of iterations and algorithm runtime  for various problem  sizes $(m,2^m)$ with different parameter $\gamma$.}
	\label{Hilber-m-3m}
	{\small	\begin{tabular}{@{}l r  r  r  r  r  r  r  r  r  r  r@{}}
			\hline\noalign{\smallskip}
			 $m$ $\rightarrow$ &     \multicolumn{2}{c}{$7$} & \multicolumn{2}{c}{$8$} &   \multicolumn{2}{c}{$9$} & \multicolumn{2}{c}{$10$} \\  
			\cmidrule(l){2-3}  \cmidrule(l){4-5} \cmidrule(l){6-7}  \cmidrule(l){8-9} 
			 $\gamma$ $\downarrow$       &  \#(Iters) 	& Time   		&  \#(Iters) 	&  Time   &  \#(Iters) 	&  Time  &   \#(Iters) 	&  Time  &    \\ 
				\noalign{\smallskip}\hline\noalign{\smallskip}
			0.1   &     2530 & 3.13 & 2553 & 9.88 & 2571 & 32.04 & 2585 & 125.46\\
			0.3  &    2436 & 2.84 & 2469 & 10.69 & 2494 & 32.28 & 2514 & 122.38\\
			0.5 &     2357 & 2.96 & 2398 & 9.22 & 2430 & 37.43 & 2455 & 132.17\\
			0.7  &     2288 & 2.60 & 2335 & 8.87 & 2372 & 38.08 & 2402 & 114.45\\
			0.9  &     2224 & 2.43 & 2278 & 9.51 & 2320 & 29.14 & 2354 & 109.55	\\ 	
			\hline\noalign{\smallskip}
		\end{tabular}
	}
\end{table}

\subsection{Generalized Heron Problems}
The {\it classical Heron problem} was introduced by Heron of Alexandria, focuses on finding a point on a given straight line in a plane such that the sum of distances from it to two given points is minimal.  In this experiment, we consider the {\it generalized Heron problem} of finding a point that minimizes the sum of the distances to given closed convex target sets $C_i \subset \mathbb{R}^n$, $i=1,...,m$ over a system of homogeneous linear equations.

The generalized Heron problem in the context of distance functions is formulated as follows:
\begin{eqnarray*}%
	\begin{array}{ll}
		\textrm{minimize}\indent \sum_{i=1}^m\dist(x,C_i)\\
		\textrm{subject to}\indent \mathbf{A}x=\mathbf{0}_{\R^n},
	\end{array}%
\end{eqnarray*}
where $C_i \subset \mathbb{R}^n$, $i=1,...,m$ are nonempty closed convex subsets and $\mathbf{A}\in \mathbb{R}^{n\times n}$ is a matrix. 
Note that if the number of target sets $m$ is large, then it is possible that the problem has  a  number of solutions and the coordinates $x_i$ may take large values. To overcome this situation, one can have an adding term to penalize the sum of squares of each coordinate, i.e., the square of $\ell_2$-norm. From this notice, the considered problem can be read as follows:
\begin{eqnarray}\label{heron}%
\begin{array}{ll}
\textrm{minimize}\indent F(x):=\sum_{i=1}^m\dist(x,C_i)+\|x\|_2^2\\
\textrm{subject to}\indent x\in\argmin \frac{1}{2}\|\mathbf{A}x\|_2^2.
\end{array}%
\end{eqnarray}

We observe that (\ref{heron})  fits into the framework considered in (\ref{opt-plus}) when setting $f_i(x)=\dist(x,C_i)$, $i=1,...,m$, $f_{m+1}(x)=\|x\|_2^2$, and $g(x)=\frac{1}{2}\|\mathbf{A}x\|_2^2$ for all $x\in\R^n$.

In this experiment we solve a number of randomized problems where the closed convex sets $C_i \subset \mathbb{R}^n$, $i=1,...,m$  are the unit balls with the centers are created randomly in the interval $(-n^2,n^2)$. We generate all elements of the matrix $\mathbf{A}$  randomly from the interval $(-10,10)$. The MATLAB calculations are performed by GFBP (Algorithm \ref{algorithm-IFB-Plus}) with the sequences  $\alpha_k=1/k$ and $\beta_k=0.9k$ for all $k\geq1$.  In all experiments, we terminate GFBP (Algorithm \ref{algorithm-IFB-Plus}) when the relative changes between two consecutive points becomes sufficiently small, i.e.,
$$\max\left\{\frac{|F(x_k)-F(x_{k-1})|}{|F(x_{k-1})|},\frac{|g(x_k)-g(x_{k-1})|}{|g(x_{k-1})|}\right\}\leq 10^{-5}.$$
We use 10 samplings for different randomly chosen  unit balls and starting point, and the results are averaged.

Table \ref{error--tb}  shows the number of iterations  and elapsed time when GFBP (Algorithm \ref{algorithm-IFB-Plus}) satisfies the relative changes $10^{-5}$. As shown in Table \ref{error--tb}, for the same dimension, the larger number of target sets needs a longer time. However, for the same number of target sets, the number of iterations and time seem to be not significantly different even if the dimensions are different.

\begin{table}
	\centering
	\setlength{\tabcolsep}{12pt}
	\caption{\label{error--tb}Behaviors of GFBP on  a generalized Heron problem.}
	{\small \begin{tabular}{@{}l   l  r  r r  r r  r r  r r @{}}
				\hline\noalign{\smallskip}
			Dimension   	  &        \#(Target)   &  \#(Iters)  	& Time  	    			 \\ \midrule
			2                      &       10                 & 199963	& 28.57  \\
			&       50                 & 199863	  & 133.14  \\
			&       100                &199732	  & 259.07 \\
			&       500               & 198657	 & 1268.07  \\ \midrule
			3                      &       10                  & 199973	& 29.07 \\
			&        50                & 199926	     & 135.79 \\
			&        100               & 199855	    & 268.34 \\
			&        500              & 199367	    & 1309.97   \\ \midrule
			5                      &       10                  & 199985 &	30.55  \\
			&        50                & 199970	      & 137.48 \\
			&        100               & 199952	     & 268.39  \\
			&        500              & 199826	    & 1310.53  \\ 
			\hline\noalign{\smallskip}
		\end{tabular}
	}
\end{table}

	\section{Conclusions}
	
	We introduce a novel splitting method, called a generalized forward-backward method with penalty term GFBP, for finding a zero of the sum of a number of maximally monotone operators and the normal cone to the zero of another maximally monotone operator. The advantage of our method is that it allows us not only to compute the resolvent of each operator separately but also to consider a general sense of the constrained set. We provide theorems for  guaranteeing  the convergence of the method. Consequently, we apply GFBP to the minimization of the large-scale hierarchical minimization problem that is the sum of both smooth and nonsmooth convex functions subject to the set of minima of another differentiable convex function. Finally, we propose two numerical experiments on large-scale convex minimization concerning  elastic net  and generalized Heron location problems.

\begin{acknowledgements}
The authors are thankful to two anonymous referees and the Associate Editor for comments and remarks which improved the quality of the paper.

\end{acknowledgements}



\end{document}